\RequirePackage[l2tabu,orthodox]{nag}
\documentclass[reqno,12pt,a4paper,oneside]{amsart}
\usepackage
           {geometry}
\usepackage{ifxetex,ifluatex}
\newif\ifxetexorluatex
\ifxetex
  \xetexorluatextrue
\else
  \ifluatex
    \xetexorluatextrue
  \else
    \xetexorluatexfalse
  \fi
\fi

\ifxetexorluatex
  \usepackage{fontspec}           
  \usepackage{unicode-math}       
  \usepackage{polyglossia}        
  \setmainlanguage{english}
\else
  \usepackage[utf8]{inputenc}     
  \usepackage[english]{babel}     

  \usepackage{amsfonts}
  \usepackage{amssymb}
  \newcommand{\nsubset}{\not\subset}
  \usepackage{lmodern}
  \usepackage[T1]{fontenc}
\fi
\usepackage{fixltx2e}   
\usepackage{booktabs}   
\usepackage{enumitem}   
\usepackage{xspace}     
\usepackage{xcolor}     

\usepackage{amsmath}    
\usepackage{mathtools}  
\usepackage{stmaryrd} 

\usepackage{graphics}
\usepackage[all,rotate]{xy} 

\definecolor{linkblue}{RGB}{1,1,190}
\definecolor{citegreen}{RGB}{1,190,1}
\usepackage[linkcolor=linkblue
           ,citecolor=citegreen
           ,ocgcolorlinks        
           ,bookmarksopen=True,  
           ]{hyperref}
\makeatletter
  \AtBeginDocument{
    \hypersetup{
      pdftitle  = {\@title},
      pdfauthor = {\authors}     
    }
  }
\makeatother

\usepackage{amsthm}

\usepackage[capitalize    
           ]{cleveref}

\theoremstyle{definition}
\newtheorem {defi}          {Definition}[section]

\theoremstyle{plain}
\newtheorem {thm}     [defi]{Theorem}
\newtheorem*{thm*}          {Theorem}
\newtheorem {lemma}   [defi]{Lemma}
\newtheorem {prop}    [defi]{Proposition}

\theoremstyle{remark}
\newtheorem {remark}  [defi]{Remark}

\newtheorem*{exm*}          {Example}
\makeatletter

\newcommand{\sc@lettershortcut}[3]{%
  \expandafter\providecommand\csname #2#3\endcsname{#1{#3}}%
}

\newcommand{\sc@shortcuts}[3]{%
  \count@=0
  \loop
  \advance\count@ 1
  \edef\tmp@{%
    \noexpand\sc@lettershortcut\unexpanded{{#1}}{#2}{#3\count@}
  }
  \tmp@
  \ifnum\count@<26
  \repeat
}

\newcommand{\defshortcuts}[2]{\sc@shortcuts{#1}{#2}{\@alph}}

\newcommand{\defShortcuts}[2]{\sc@shortcuts{#1}{#2}{\@Alph}}

\makeatother
\defShortcuts{\mathbb}{b}
\defShortcuts{\mathcal}{c}
\defShortcuts{\mathfrak}{f}
\defShortcuts{\mathsf}{s}
\defshortcuts{\mathfrak}{f}
\defshortcuts{\mathsf}{s}

\def\rfop{*}
\makeatletter
\newcommand\rigidfactorization[2][]{%
  \def\rf@delim{\rfop}
  \newif\ifrf@notfirst
  #1
  \@for\next:=#2\do{%
    \ifrf@notfirst
      \rf@delim
    \fi
    \rf@notfirsttrue
    \next
  }%
}
\makeatother
\newcommand\rf\rigidfactorization
\newcommand{\quo}{\mathsf{q}}                     

\ifxetexorluatex
  
\else
  
\fi
\providecommand{\val}{\mathsf{v}}                 

\newcommand{\cc}[2]{{({#1}\!:\!{#2})}}          

\DeclarePairedDelimiter{\card}{\lvert}{\rvert}

\DeclarePairedDelimiter{\kcls}{\langle}{\rangle}

\DeclareMathOperator{\End}{End}
\DeclareMathOperator{\Aut}{Aut}

\DeclareMathOperator{\id}{id}

\DeclareMathOperator{\spec}{spec}

\DeclareMathOperator{\udim}{udim}
\DeclareMathOperator{\height}{ht}
\DeclareMathOperator{\supp}{supp}

\newcommand{\pone}{\mathfrak X}

\setlist[enumerate,1]{label=\textup{(\arabic*)}, ref=(\arabic*), leftmargin=0.75cm}
\setlist[enumerate,2]{label=(\roman*), ref=(\roman*)}

\newenvironment{equivenumerate}[1][]
{\enumerate[label=\textup{(\alph*)},ref=(\alph*),#1]}
{\endenumerate}

\makeatletter
\def\sr@stripleadingcol::#1{#1}
\def\sr@dosubref#1#2:#3 #4{\if\relax#3\relax%
  \def\first{\sr@stripleadingcol #4}%
  #1{\first}\ref{\first:#2}%
\else%
  \sr@dosubref#1#3 {#4:#2}%
\fi}%

\newcommand{\subref}[1]{\sr@dosubref\cref#1: :\relax}
\newcommand{\Subref}[1]{\sr@dosubref\Cref#1: :\relax}
\makeatother
\AtBeginDocument{\renewcommand{\setminus}{\smallsetminus}}

\usepackage{microtype}

\title{Every abelian group is the class group of a simple Dedekind domain}
\author[D.~Smertnig]{Daniel Smertnig}

\address{University of Graz\\
         NAWI Graz\\
         Institute for Mathematics and Scientific Computing\\
         Heinrichstra\ss e 36\\
         8010 Graz, Austria}
\email{daniel.smertnig@uni-graz.at}
\thanks{The author was supported by the Austrian Science Fund (FWF) project P26036-N26.}


\begin{document}

\begin{abstract}
A classical result of Claborn states that every abelian group is the class group of a commutative Dedekind domain.
Among noncommutative Dedekind prime rings, apart from PI rings, the simple Dedekind domains form a second important class.
We show that every abelian group is the class group of a noncommutative simple Dedekind domain.
This solves an open problem stated by Levy and Robson in their recent monograph on hereditary Noetherian prime rings.
\end{abstract}

\maketitle

\section{Introduction}

Throughout the paper, a domain is a not necessarily commutative unital ring in which the zero element is the unique zero divisor.
In \cite{claborn66}, Claborn showed that every abelian group $G$ is the class group of a commutative Dedekind domain. An exposition is contained in \cite[Chapter III \S14]{fossum73}.
Similar existence results, yielding commutative Dedekind domains which are more geometric, respectively number theoretic, in nature, were obtained by Leedham-Green in \cite{leedham-green72} and Rosen in \cite{rosen73,rosen76}.
Recently, Clark in \cite{clark09} showed that every abelian group is the class group of an elliptic commutative Dedekind domain, and that this domain can be chosen to be the integral closure of a PID in a quadratic field extension.
See Clark's article for an overview of his and earlier results.
In commutative multiplicative ideal theory also the distribution of nonzero prime ideals within the ideal classes plays an important role.
For an overview of realization results in this direction see \cite[Chapter 3.7c]{ghk06}.

A ring $R$ is a Dedekind prime ring if every nonzero submodule of a (left or right) progenerator is a progenerator (see \cite[Chapter 5]{mcconnell-robson01}).
Equivalently, $R$ is a hereditary Noetherian prime ring which is also a maximal order in its simple Artinian quotient ring.
A Dedekind domain is a Dedekind prime ring $R$ which is also a domain (equivalently, $\udim R_R = \udim {}_R R = 1$).
To a Dedekind prime ring $R$ one can associate an (abelian) class group $G(R)$ in such a way that $K_0(R) \cong G(R) \times \bZ$.
Equivalently, $G(R)$ can also be interpreted as a group of stable isomorphism classes of essential right ideals of $R$.
Since $K_0$ is Morita invariant, the same holds for the class group.
Every Dedekind prime ring is Morita equivalent to a Dedekind domain.

Realization questions for class groups within the class of strictly noncommutative Dedekind prime rings have an easy answer.
If $R$ is a commutative Dedekind domain with class group $G(R)$ and $M$ is a finitely generated projective $R$-module, then $S=\End_R(M)$ is a Dedekind prime ring with $G(S) \cong G(R)$.
However, $S$ is a PI ring, and thus in many aspects close to being commutative.

On the other hand, there exist Dedekind prime rings (and domains) of a very different nature.
For instance, the first Weyl algebra $A_1(K)$ over a field $K$ of characteristic $0$ is a simple Dedekind domain with trivial class group.
The ring $R=\bR[X,Y]/(X^2+Y^2-1)$ is a commutative Dedekind domain with $G(R) \cong \bZ/2\bZ$.
If $\sigma \in \Aut(R)$ denotes the automorphism induced by the rotation by an irrational angle, then the skew Laurent polynomial ring $T = R[x,x^{-1};\sigma]$ is a noncommutative Dedekind domain with $G(T) \cong \bZ/2\bZ$.
Similar constructions exist that show that $\bZ^n$ for $n \in \bN_0$ appears as class group of a noncommutative Dedekind prime ring.
(See \cite[\S7.11 and \S12.7]{mcconnell-robson01} for details.)

The mentioned rings are not Morita equivalent to commutative Dedekind domains and they are all simple rings.
In fact, in \cite{goodearl-stafford05}, a striking dichotomy is established: A Dedekind domain which is finitely generated as an algebra over $\bC$ is commutative or simple.
More generally, if $K$ is a field and a $K$-algebra $R$ is a Dedekind prime ring such that $\dim_K R < \card{K}$ and $R \otimes_K \overline K$ is Noetherian, then $R$ is a PI ring or simple.

In \cite[Problem 54.7]{levy-robson11}, Levy and Robson state it as an open problem to determine which abelian groups can appear as class groups of simple Dedekind prime rings.
The present paper answers this question by showing that any abelian group can be realized as the class group of a simple Dedekind domain.
The main theorem we prove is the following.

\begin{thm} \label{t-main}
  Let $G$ be an abelian group, $K$ a field, and $\kappa$ a cardinal.
  Then there exists a $K$-algebra $T$ which is a noncommutative simple Dedekind domain, $G(T)\cong G$, and each class of $G(T)$ contains at least $\kappa$ maximal right ideals of $T$.
\end{thm}

Simple noncommutative Dedekind domains are canonically obtained either as skew Laurent polynomial rings $R[x,x^{-1};\sigma]$ or as skew polynomial rings $R[x;\delta]$, where $R$ is a commutative Dedekind domain and $\sigma$ is an automorphism, respectively $\delta$ a derivation.
The domains we construct are skew Laurent polynomial rings.
It is well understood how class groups behave under this extension.
In this way, the problem reduces to the construction of a commutative Dedekind domain $R$ with prescribed class group and automorphism $\sigma$ of $R$.
The automorphism $\sigma$ must be such that no proper nonzero ideal $\fa$ is $\sigma$-stable (that is, $\sigma(\fa) = \fa$), but such that the induced automorphism on the class group of $R$ is trivial.

The actual construction is very conceptual in nature and proceeds through the following steps:

\begin{enumerate}
  \item Construct a commutative Krull monoid with class group $G$ and a monoid automorphism $\tau$ of $H$ such that no nonempty proper divisorial ideal of $H$ is $\tau$-stable.
  \item Extend $\tau$ to $K[H]$.
    The semigroup algebra $K[H]$ is a commutative Krull domain with class group isomorphic to $G$.
    The crucial step lies in establishing that no nonzero proper divisorial ideal of $K[H]$ is $\tau$-stable.
  \item A suitable localization $R=S^{-1}K[H]$ is a commutative Dedekind domain, has the same class group as $K[H]$, and $\tau$ extends to $R$.
    This is analogous to the same step in Claborn's proof.
  \item The skew Laurent polynomial ring $T=R[x,x^{-1};\tau]$ is a noncommutative simple Dedekind domain with $G(T) \cong G$.
\end{enumerate}

The methods work in greater generality.
For instance, the field $K$ can be replaced by a commutative Krull domain with suitable automorphism.
The full result is stated in \cref{t-extend}.
\Cref{t-main} is an immediate consequence of \cref{t-ex-mon-aut} and \cref{t-extend}.
The actual construction is mostly commutative in nature.
Before giving the proofs in \cref{s-proofs}, a number of preliminary results are recalled in \cref{s-prelim}.

\begin{remark}
  Let $R$ be a commutative Dedekind domain which is an affine algebra over a field $K$ of characteristic $0$.
  Then the ring of differential operators $\cD(R)$ is a simple Dedekind domain, and the inclusion $R \hookrightarrow \cD(R)$ induces an isomorphism $K_0(R) \cong K_0(\cD(R))$ (see \cite[Chapter 15]{mcconnell-robson01}).
  This induces an isomorphism $G(R) \cong G(\cD(R))$.
  In \cite{rosen73}, Rosen has shown that any finitely generated abelian group is the class group of a commutative Dedekind domain which is affine over a number field.
  This gives a different way of showing that any finitely generated abelian group is the class group of a simple Dedekind domain.
  Using the results from \cite{clark09}, this can be extended to groups of the form $F/H$ where $F$ is free abelian and $H$ is a finitely generated subgroup.
\end{remark}

\section{Background: Krull monoids and skew Laurent polynomial rings} \label{s-prelim}

All rings and modules are unital.
Ring homomorphisms preserve the multiplicative identity.
If $X$ is a subset of a domain, we set $X^\bullet=X \setminus \{0\}$.
A \emph{monoid} is a cancellative semigroup with a neutral element.
Monoid homomorphisms preserve the neutral element.
If $H$ is a monoid, $H^\times$ denotes its group of units.
$H$ is \emph{reduced} if $H^\times = \{1\}$.
A commutative monoid is \emph{torsion-free} if its quotient group is torsion-free.
$\bN$ denotes the set of positive integers and $\bN_0$ the set of all nonnegative integers.
For sets $A$ and $B$, inclusion is denoted by $A \subset B$ and strict inclusion by $A \subsetneq B$.

For a set $P$, let $\cF(P)$ denote the multiplicatively written free abelian monoid with basis $P$.
The quotient group $\quo(\cF(P))$ of $\cF(P)$ is the free abelian group with basis $P$.
Each $a \in \quo(\cF(P))$ has a unique (up to order) representation of the form $a = p_1^{n_1}\cdots p_r^{n_r}$ with $r \in \bN_0$, pairwise distinct $p_1$,~$\ldots\,$,~$p_r \in P$ and $n_1$,~$\ldots\,$,~$n_r \in \bZ^\bullet$.
We define $\supp(a) = \{p_1,\ldots, p_r\}$, $\val_{p_i}(a)=n_i$ for $i \in [1,r]$ and $\val_q(a) = 0$ for all $q \in P \setminus \supp(a)$.

\subsection{Commutative Krull monoids and commutative Krull domains}
We use \cite[Chapter 2]{ghk06} as a reference for commutative Krull monoids and \cite{fossum73} as reference for commutative Krull domains.
For semigroup algebras we refer to \cite{gilmer84}.

Let $(H,\cdot)$ be a commutative monoid and let $\tau \in \Aut(H)$.
Let $\quo(H)$ denote the quotient group of $H$.
The automorphism $\tau$ naturally extends to an automorphism of $\quo(H)$, which we also denote by $\tau$.
For subsets $X$,~$Y \subset \quo(H)$ we define $\cc{Y}{X} = \{\, a \in \quo(H) \mid aX \subset Y \,\}$.
We set $X^{-1} = \cc{H}{X}$ and $X_v = (X^{-1})^{-1}$.
Then $\tau\big({\cc{Y}{X}}\big) = \cc{\tau(Y)}{\tau(X)}$ and hence $\tau(X^{-1}) = \tau(X)^{-1}$ and $\tau(X)_v = \tau( X_v )$.

A subset $\fa \subset \quo(H)$ is a \emph{fractional ideal} of $H$ if $H\fa \subset \fa$ and there exists a $d \in H$ such that $d\fa \subset H$.
If in addition $\fa \subset H$, then $\fa$ is an \emph{ideal} of $H$.
A fractional ideal $\fa$ is \emph{divisorial} if $\fa = \fa_v$.
For all $a \in \quo(H)$, $(aH)_v = aH$ and hence principal fractional ideals are divisorial.
If $\fa$ and $\fb$ are divisorial fractional ideals of $H$, their \emph{divisorial product} is $\fa \cdot_v \fb = (\fa \cdot \fb)_v$.
For principal fractional ideals, the divisorial product coincides with the usual ideal product.

$H$ is a \emph{commutative Krull monoid} if it is $v$-Noetherian (i.e., satisfies the ascending chain condition on divisorial ideals) and completely integrally closed (i.e, whenever $x \in \quo(H)$ is such that there exists a $c \in H$ such that $cx^n \in H$ for all $n \in \bN$, then already $x \in H$).
From now on, let $H$ be a commutative Krull monoid.
If $\fa$ is a nonempty divisorial fractional ideal of $H$, then $\fa$ is invertible with respect to the divisorial product, i.e., $\fa \cdot_v \fa^{-1} = H$.
We denote by $\cF_v(H)^\times$ the group of all nonempty divisorial fractional ideals, and by $\cI_v^*(H)$ the monoid of all nonempty divisorial ideals.
Let $\pone(H)$ be the set of nonempty divisorial prime ideals.
Recall that $\pone(H)$ consists precisely of the prime ideals of height $1$.

With respect to the divisorial product, $\cI_v^*(H)$ is the free abelian monoid with basis $\pone(H)$, and $\cF_v(H)^\times$ is the free abelian group with basis $\pone(H)$.
Hence, every $\fa \in \cF_v(H)^\times$ has a unique representation of the form
\[
\fa = \fp_1^{n_1} \cdot_v \ldots \cdot_v \fp_r^{n_r}
\]
with $r \in \bN_0$, pairwise distinct $\fp_1$,~$\ldots\,$,~$\fp_r \in \pone(H)$ and $n_1$,~$\ldots\,$,~$n_r \in \bZ^\bullet$.
We have $\supp(\fa) = \{\fp_1,\ldots,\fp_r\}$ and $\val_{\fp_i}(\fa)= n_i$ for $i \in [1,r]$.

The principal fractional ideals form a subgroup of $\cF_v(H)^\times$. The \emph{class group} of $H$ is the factor group
\[
\cC(H) = \cF_v(H)^\times / \{\, aH \mid a \in \quo(H) \,\}.
\]
We use additive notation for $\cC(H)$.
If $\fa \in \cF_v(H)^\times$, we write $[\fa]=[\fa]_H$ for its class in $\cC(H)$.
If $\fa$,~$\fb \in \cF_v(H)^\times$, then $[\fa \cdot_v \fb] = [\fa] + [\fb]$.

Any $\tau \in \Aut(H)$ induces an automorphism $\tau_*$ of $\cF_v(H)^\times$ by means of $\tau_*(\fa) = \tau(\fa)$.
Then $\tau_*(\pone(H)) = \pone(H)$, the restriction $\tau_*=\tau_*|_{\cI_v^*(H)}$ is a monoid automorphism of $\cI_v^*(H)$, and $\tau_*(aH)=\tau(a)H$ for all $a \in \quo(H)$.
In particular, $\tau_*$ induces an automorphism of $\cC(H)$, also denoted by $\tau_*$, by means of $\tau_*([\fa])=[\tau_*(\fa)]$.

A \emph{commutative Krull domain} is a domain $D$ such that $D^\bullet$ is a commutative Krull monoid.
We use similar notation for Krull domains as we have introduced for Krull monoids.
If $\quo(D)$ denotes the quotient field of $D$ and $X \subset \quo(D)$, then $\cc{D}{X}$ is always additively closed.
This implies that there exists an isomorphism
\[
\cF_v(D)^\times \to \cF_v(D^\bullet)^\times,\quad \fa \mapsto \fa^\bullet.
\]
Concepts related to divisorial ideals on $D$ correspond to ones on $D^\bullet$.
We make use of this without further mention.
In particular, $\cC(D) \cong \cC(D^\bullet)$ canonically, and we identify.
A commutative domain $D$ is a Dedekind domain if and only if it is a Krull domain with $\dim(D) \le 1$.
Then every nonzero fractional ideal of $D$ is invertible and hence divisorial.
In particular, $\cC(D)$ is the usual ideal class group of the Dedekind domain.

We will construct Krull domains from Krull monoids using semigroup algebras.
The following result is essential.

\begin{prop}[{\cite[Theorem 15.6 and Corollary 16.8]{gilmer84}}] \label{p-krull-semigroup-domain}
  Let $D$ be a commutative domain and $H$ a torsion-free commutative monoid.
  The semigroup algebra $D[H]$ is a Krull domain if and only if $D$ is a Krull domain, $H$ is a Krull monoid, and $H^\times$ satisfies the ascending chain condition on cyclic subgroups.
  In this case $\cC(D[H]) \cong \cC(D) \times \cC(H)$.
\end{prop}

The isomorphism between $\cC(D) \times \cC(H)$ and $\cC(D[H])$ is obtained naturally by extending representatives of the divisorial ideal classes in $D$, respectively $H$, to $D[H]$.
If $\fa$ is a fractional ideal of $D$, let $\fa[H]=\fa D[H]$ be the extension of $\fa$ to $D[H]$.
It consists of all elements all of whose coefficients are contained in $\fa$.
If $\fb$ is a fractional ideal of $H$, let $D[\fb]= \fb D[H]$ be the extension of $\fb$ to $D[H]$.
It consists of all elements whose support is contained in $\fb$.
By $\fa[\fb]$ we denote the fractional ideal whose support is contained in $\fb$ and whose coefficients are contained in $\fa$.
Then $\fa[\fb] = \fa[H] \cdot D[\fb]$.
Explicitly, the isomorphism of class groups is given by
\[
\cC(D) \times \cC(H) \to \cC(D[H]), \quad
([\fa]_D, [\fb]_H)   \mapsto \big[ \fa[\fb] \big]_{D[H]}.
\]
Let $\sigma \in \Aut(D)$, $\tau \in \Aut(H)$ and let $\varphi \in \Aut(D[H])$ be the extension of $\sigma$ and $\tau$ to $D[H]$ (i.e., $\varphi|_D=\sigma$ and $\varphi|_H = \tau$).
Under the stated isomorphism of the class groups, the automorphism $(\sigma_*, \tau_*)$ on $\cC(D) \times \cC(H)$ corresponds to $\varphi_*$ on $\cC(D[H])$.
From now on we identify $\cC(D[H]) \cong \cC(D) \times \cC(H)$.

\begin{prop}[Nagata's Theorem, {\cite[Corollary 7.2]{fossum73}}] \label{p-nagata}
  Let $D$ be a commutative Krull domain and $S \subset D^\bullet$ a multiplicative subset.
  Then the localization $S^{-1}D$ is a Krull domain and the map $\cI_v^*(D) \to \cI_v^*(S^{-1}D)$, $\fa \mapsto S^{-1}\fa$ induces an epimorphism $\cC(D) \to \cC(S^{-1}D)$ with kernel generated by those $\fp \in \pone(D)$ with $\fp \cap S \ne \emptyset$.
  In particular, if $S$ is generated by prime elements of $D$, then $\cC(D) \cong \cC(S^{-1}D)$.
\end{prop}

Let $D$ be a commutative Krull domain and let $S \subset D^\bullet$ be a multiplicative subset.
Then $S^{-1}D$ is a Dedekind domain if and only if $\dim(S^{-1}D) \le 1$.
This is the case if and only if $S \cap \fP \ne \emptyset$ for all $\fP \in \spec(D)$ with $\height(\fP) > 1$.

\subsection{Skew Laurent polynomial rings}

Let $R$ be a ring and $\sigma \in \Aut(R)$.
By $R[x,x^{-1};\sigma]$ we denote the ring of \emph{skew Laurent polynomials}.
$R[x,x^{-1};\sigma]$ consists of polynomial expressions in $x$ and $x^{-1}$ with coefficients in $R$ and subject to $ax=x\sigma(a)$ for all $a \in R$.
Let $\fa$ be an ideal of $R$.
If $\sigma \in \Aut(R)$, then $\fa$ is \emph{$\sigma$-stable} if $\sigma(\fa) = \fa$.
The ring $R$ is \emph{$\sigma$-simple} if $\mathbf 0$ and $R$ are the only $\sigma$-stable ideals of $R$.

\begin{prop}[{\cite[Theorem 1.8.5]{mcconnell-robson01}}]
  Let $R$ be a ring, $\sigma \in \Aut(R)$ and $T = R[x,x^{-1};\sigma]$.
  Then $T$ is a simple ring if and only if $R$ is $\sigma$-simple and no power of $\sigma$ is an inner automorphism.
\end{prop}

If $R$ is a commutative ring, the identity is the only inner automorphism of $R$.
Hence the second condition in the previous theorem reduces to $\sigma$ having infinite order.
If $R$ is a $\sigma$-simple commutative domain which is not a field, then $\sigma$ has infinite order.
For suppose $\sigma^n = \id$ for some $n \in \bN$.
Let $\mathbf 0 \ne \fa \subsetneq R$ be an ideal of $R$.
Then $\fa\sigma(\fa) \cdots \sigma^{n-1}(\fa) \ne \mathbf 0$ is a proper ideal of $R$ which is $\sigma$-stable.

Combining our observations so far with \cite[Theorem 7.11.2]{mcconnell-robson01}, we obtain the following.
\begin{prop}[{\cite[Theorem 7.11.2]{mcconnell-robson01}}] \label{p-nc-dedekind}
  Let $R$ be a commutative Dedekind domain which is not a field, let $\sigma \in \Aut(R)$, and let $T=R[x,x^{-1};\sigma]$.
  The following conditions are equivalent:
  \begin{equivenumerate}
    \item $T$ is simple.
    \item $T$ is hereditary.
    \item The Krull dimension of $T$ is $1$.
    \item $T$ is a noncommutative Dedekind domain.
    \item $R$ is $\sigma$-simple.
  \end{equivenumerate}
\end{prop}

The behavior of the Grothendieck group $K_0$ under skew Laurent polynomial extensions is well understood.
We denote classes in $K_0$ using angle brackets.
We recall the result from \cite[\S12.5]{mcconnell-robson01}.
Let $R$ be a ring and $\sigma \in \Aut(R)$.
Let $M$ be a right $R$-module.
Define a new right $R$-module $M^\sigma$ as follows:
As a set, $M^\sigma$ is in bijection with $M$, where the element of $M^\sigma$ corresponding to $m \in M$ is written as $m^\sigma$.
The abelian group structure on $M^\sigma$ is the one induced from $M$, i.e., $m^\sigma + n^\sigma = (m+n)^\sigma$.
The right $R$-module structure on $M^\sigma$ is defined by $(m^\sigma) r = (m \sigma^{-1}(r))^\sigma$.
A similar construction works for left modules: To a left module $M$ associate ${}^\sigma M$ with $r ({}^\sigma m) = {}^\sigma (\sigma(r) m)$.
In particular, ${}^\sigma R$ with the usual right $R$-module structure is an $R$-bimodule, and $M \otimes_R ({}^\sigma R) \cong M^\sigma$ as right $R$-modules.
Now, $\sigma$ induces an automorphism $\sigma_*$ of $K_0(R)$ by means of $\kcls{M} \mapsto \kcls{M^\sigma}$.

Let $T = R[x,x^{-1};\sigma]$.
For a finitely generated projective right $R$-module $M$, $M \otimes_R T$ is a finitely generated projective right $T$-module.
This induces a homomorphism $\alpha\colon K_0(R) \to K_0(T)$.
A ring $R$ is right regular if each finitely generated right $R$-module has a projective resolution of finite length.

\begin{prop}[{\cite[Theorem 12.5.6]{mcconnell-robson01}}] \label{p-k0-ext}
  Let $R$ be a right regular, right Noetherian ring.
  Let $\sigma \in \Aut(R)$ and $T = R[x,x^{-1};\sigma]$.
  Then the sequence
  \[
  \xymatrix@C=1.5cm{
    K_0(R) \ar[r]^{\id-\sigma_*} & K_0(R) \ar[r]^{\alpha} & K_0(T) \ar[r] & \mathbf 0
  }
  \]
  is exact.
\end{prop}

Let $R$ be a Dedekind prime ring.
Each finitely generated projective right $R$-module $P$ has a uniform dimension $\udim_R (P) \in \bN_0$.
The uniform dimension is additive on direct sums and induces an epimorphism $\udim_R \colon K_0(R) \to \bZ$.
The \emph{(ideal) class group} of $R$ is $G(R) = \ker(\udim_R\colon K_0(R) \to \bZ)$.
The epimorphism $\udim_R$ splits, hence $K_0(R) \cong G(R) \times \bZ$.
Let $G'$ denote the set of stable isomorphism classes of essential right ideals of $R$.
$G'$ can be endowed with the structure of an abelian group by setting $[\fa] + [\fb] = [\fc]$ if and only if $\fa \oplus \fb \cong R \oplus \fc$.
Then $G'$ is isomorphic to $G(R)$ by means of $G' \to G(R), [\fa] \mapsto \kcls{\fa} - \kcls{R}$, and we identify.
When we say that a class $g \in G(R)$ contains an essential right ideal $\fa$ of $R$, we mean $g=[\fa]=\langle \fa \rangle - \langle R \rangle$.

If $R$ is commutative, $G(R)$ is indeed isomorphic the usual ideal class group.
The isomorphism $\cC(R) \to G(R)$ is given by $[\fa] \mapsto \kcls{\fa} - \kcls{R}$.
If $\sigma$ is an automorphism of $R$, we note that under the stated isomorphism of $\cC(R)$ and $G(R)$, the induced automorphism $\sigma_* \colon \cC(R) \to \cC(R)$ corresponds to $\sigma_* \colon G(R) \to G(R)$.
This is so, because for an ideal $\fa \subset R$, we have $\fa^\sigma \cong \sigma(\fa)$ as right $R$-modules, via $a^\sigma \mapsto \sigma(a)$.

Let $R$ be a commutative Dedekind domain.
Since $\udim_T(P \otimes_R T) = \udim_R(P)$ for all finitely generated projective $R$-modules $P$, we obtain a commutative diagram
\[
  \xymatrix@C=1.5cm{
    K_0(R) \ar[r]^{\id-\sigma_*} \ar[d] & K_0(R) \ar[r]^{\alpha} \ar[d] & K_0(T) \ar[r] \ar[d] & \mathbf 0 \\
    G(R) \times \bZ \ar[r]^{(\id-\sigma_*,\id)} & G(R) \times \bZ \ar[r]^{(\alpha_0,\id)} & G(T) \times \bZ \ar[r] & \mathbf 0 \\
  }
\]
with the vertical arrows being isomorphisms induced by the splitting of $\udim_R$, $\udim_R$, and $\udim_T$ respectively.
Here $\alpha_0$ is the map induced on $G(R) \to G(T)$ by $\alpha$.
Using the isomorphism $\cC(R) \cong G(R)$, we obtain a short exact sequence
\[
  \xymatrix@C=1.5cm{
    \cC(R) \ar[r]^{\id-\sigma_*} & \cC(R) \ar[r]^{\beta} & G(T) \ar[r] & \mathbf 0.
  }
\]
Here, $\beta([\fa]_R) = \kcls{\fa \otimes_R T} - \kcls{T} = [\fa \otimes_R T]_T \in G(T)$.

\begin{remark}
  \begin{enumerate}
  \item
    If $R$ and $S$ are Morita equivalent Dedekind prime rings, the Morita equivalence induces an isomorphism $K_0(R) \cong K_0(S)$, which restricts to an isomorphism $G(R) \cong G(S)$.
  \item
  Let $R$ be a Dedekind prime ring.
  If $\fa$ and $\fb$ are stably isomorphic essential right ideals of $R$, that is $[\fa]=[\fb]$ in $G(R)$, then, in general, it does not follow that $\fa \cong \fb$.
  However, if $\udim_R R \ge 2$, then $[\fa]=[\fb]$ does imply $\fa \cong \fb$ (\cite[Corollary 35.6]{levy-robson11}).
  Note that $S=M_n(R)$, with $n \ge 2$ is a Dedekind prime ring with $G(R) \cong G(S)$ and $\udim_S S =n \ge 2$.
\end{enumerate}
\end{remark}

\section{Construction and main results} \label{s-proofs}

\begin{lemma}
  Let $H$ be a commutative Krull monoid, $\tau \in \Aut(H)$, and $\fa \in \cF_v^\times(H)$.
  Then $\tau(\fa) = \fa$ if and only if $\tau(\fa) \subset \fa$.
\end{lemma}

\begin{proof}
  Suppose that $\tau(\fa) \subset \fa$.
  Then there exists $\fb \in \cI_v^*(H)$ such that $\tau(\fa) = \fa \cdot_v \fb$.
  Let $\fa = \fp_1^{n_1} \cdot_v \ldots \cdot_v \fp_r^{n_r}$ with $r \in \bN_0$, $\fp_1$,~$\ldots\,$,~$\fp_r \in \pone(H)$ and $n_1$,~$\ldots\,$,~$n_r \in \bZ^\bullet$.
  Similarly, let $\fb = \fq_1^{m_1} \cdot_v \ldots \cdot_v \fq_s^{m_s}$ with $s \in \bN_0$, $\fq_1$,~$\ldots\,$,~$\fq_s \in \pone(H)$ and $m_1$,~$\ldots\,$,~$m_s \in \bN$.
  Then
  \[
  \tau(\fa) = \tau(\fp_1)^{n_1} \cdot_v \ldots \cdot_v \tau(\fp_r)^{n_r} = \fp_1^{n_1} \cdot_v \ldots \cdot_v \fp_r^{n_r} \cdot_v \fq_1^{m_1} \cdot_v \ldots \cdot_v \fq_s^{m_s}.
  \]
  Then necessarily $n_1+\cdots + n_r = n_1 + \cdots + n_r + m_1+ \cdots + m_s$.
  Hence $s=0$ and $\fb=H$.
  Thus $\tau(\fa) = \fa$.
\end{proof}

Of course, the claim of the previous lemma does not hold for ideals which are not divisorial.
For a counterexample, let $K$ be a field, $H=K[...,X_{-1},X_0,X_1,\ldots]^\bullet$, $\tau(X_i) = X_{i+1}$ with $\tau|_K=\id$, and $\fa = (X_0,X_1,\ldots)$.

\begin{lemma} \label{l-simple}
  Let $H$ be a commutative Krull monoid and let $\tau \in \Aut(H)$.
  The following statements are equivalent:
  \begin{equivenumerate}
    \item\label{l-simple:fracideal} $\tau(\fa) \ne \fa$ for all $\fa \in \cF_v(H)^\times \setminus \{H\}$.
    \item\label{l-simple:ideal} $\tau(\fa) \ne \fa$ for all $\fa \in \cI_v^*(H) \setminus \{H\}$.
    \item\label{l-simple:sqf} $\tau(\fa) \ne \fa$ for all squarefree $\fa \in \cI_v^*(H) \setminus \{H\}$.
    \item\label{l-simple:prime} For all finite $\emptyset \ne X \subset \pone(H)$, it holds that $\tau_*(X) = \{\, \tau(\fp) \mid \fp \in X \,\} \ne X$.
    \item\label{l-simple:orbits} The induced permutation $\tau_*$ of $\pone(H)$ has no finite orbits.
  \end{equivenumerate}
  If $\cC(H)=\mathbf 0$, then any of the above conditions is equivalent to
  \begin{equivenumerate}
    \setcounter{enumi}{5}
    \item\label{l-simple:principal} For all $a \in H\setminus H^\times$ and $\varepsilon \in H^\times$, $\tau(a) \ne \varepsilon a$.
  \end{equivenumerate}

  In particular, if these equivalent conditions are satisfied and $\emptyset \ne A \subset \quo(H)$ is finite with $A \nsubset H^\times$, then $\tau(A) \ne A$.
\end{lemma}

\begin{proof}
  \ref*{l-simple:fracideal}${}\Rightarrow{}$\ref*{l-simple:ideal}${}\Rightarrow{}$\ref*{l-simple:sqf}: Trivial.

  \ref*{l-simple:sqf}${}\Rightarrow{}$\ref*{l-simple:prime}:
  By contradiction. Suppose that $\emptyset \ne X \subset \pone(H)$ is such that $\tau_*(X)=X$.
  Set $\fa = ( \prod_{\fp \in X} \fp )_v$.
  Then $\fa \in \cI_v^*(H) \setminus \{H\}$, $\fa$ is squarefree, and $\tau(\fa) = \tau\big( ( \prod_{\fp \in X} \fp )_v \big ) = ( \prod_{\fp \in X} \tau(\fp) )_v = \fa$.
  This contradicts \ref*{l-simple:sqf}.

  \ref*{l-simple:prime}${}\Rightarrow{}$\ref*{l-simple:orbits}: Clear.

  \ref*{l-simple:orbits}${}\Rightarrow{}$\ref*{l-simple:fracideal}:
  Let $\fa \in \cF_v^*(H) \setminus \{H\}$.
  Then $\fa = \fp_1^{n_1}\cdot_v \ldots \cdot_v \fp_r^{n_r}$ with $r \in \bN$, $\fp_1$, $\ldots\,$,~$\fp_r \in \pone(H)$ and $n_1$, $\ldots\,$,~$n_r \in \bZ^\bullet$.
  Now $\tau(\fa) = \tau(\fp_1)^{n_1} \cdot_v \ldots \cdot_v \tau(\fp_r)^{n_r}$ is the unique representation of $\tau(\fa)$ as divisorial product of divisorial prime ideals.
  Suppose that $\tau(\fa)=\fa$.
  Then $\tau^n(\fa) = \fa$ for all $n \in \bZ$.
  Hence the $\tau_*$-orbit of $\fp_1$ is contained in $\supp(\fa) = \{ \fp_1,\ldots,\fp_r \}$.
  This contradicts \ref*{l-simple:orbits}.

  \ref*{l-simple:ideal}${}\Leftrightarrow{}$\ref*{l-simple:principal}
  Suppose that $\cC(H)$ is trivial.
  Then every divisorial ideal is principal.
  The claim follows since $aH = bH$ for $a$,~$b \in H$ if and only if there exists $\varepsilon \in H^\times$ with $a = b \varepsilon$.

  We still have to show the final implication and do so by contradiction.
  Let $\emptyset \ne A=\{a_1, \ldots, a_n\} \subset \quo(H)$ with $A \nsubset H^\times$.
  Since $A \nsubset H^\times$, the set $X = \bigcup_{i=1}^n \supp(a_i H) \subset \pone(H)$ is nonempty.
  Thus $\tau_*(X) \ne X$ by \ref*{l-simple:prime}, and hence $\tau(A) \ne A$.
\end{proof}

\begin{defi}
  Let $H$ be a commutative Krull monoid and $\tau \in \Aut(H)$.
  $H$ is called \emph{$\tau$-$v$-simple} if the equivalent conditions of \cref{l-simple} are satisfied.
  If $D$ is a commutative Krull domain and $\sigma \in \Aut(D)$, then $D$ is called \emph{$\sigma$-$v$-simple} if the commutative Krull monoid $D^\bullet$ is $(\sigma|_{D^\bullet})$-$v$-simple.
\end{defi}

A lemma analogous to \cref{l-simple} holds for commutative Krull domains.
Since there is a correspondence between divisorial ideals of $D$ and divisorial ideals of $D^\bullet$, $D$ is $\sigma$-$v$-simple if and only if $\sigma(\fa) \ne \fa$ for all divisorial ideals $\fa$ of $D$, etc.

We first construct a reduced commutative Krull monoid $H$ with given class group $G$, as well as an automorphism of $H$ such that $H$ is $\tau$-$v$-simple, and such that $\tau_*$ acts trivially on the class group.

\begin{thm} \label{t-ex-mon-aut}
  Let $G$ be an abelian group and $\kappa$ an infinite cardinal.
  Then there exists a reduced commutative Krull monoid $H$ and an automorphism $\tau$ of $H$ such that $\cC(H) \cong G$, $\tau_*=\id_{\cC(H)}$, and $H$ is $\tau$-$v$-simple.
  Each class of $\cC(H)$ contains $\kappa$ nonempty divisorial prime ideals.
\end{thm}

\begin{proof}
Let $(G,+)$ be an additive abelian group, and let $\Omega$ be a set of cardinality $\kappa$.
Let $\tau_0 \colon \Omega \to \Omega$ be a permutation such that $\tau_0(X) \ne X$ for all finite $\emptyset \ne X \subset \Omega$.
(Such a permutation always exists. $\Omega$ is in bijection with $\Omega \times \bZ$, and the map $\Omega \times \bZ \to \Omega \times \bZ$, $(x,n) \mapsto (x,n+1)$ has the desired property.)

Let $D = \cF(\Omega \times G)$ be the free abelian monoid with basis $\Omega \times G$.
Then $\tau_0$ induces an automorphism $\tau \in \Aut(D)$ with the property that $\tau((x,g)) = (\tau_0(x),g)$ for all $x \in \Omega$ and $g \in G$.
Let $\psi\colon D \to G$ be the unique homomorphism such that $\psi((x,g)) = g$ for all $x \in \Omega$ and $g \in G$.
Set $H = \psi^{-1}(0_G)$.
Since $\psi(\tau((x,g))) = g = \psi((x,g))$, we find $\tau(H) \subset H$.
Hence $\tau$ restricts to an automorphism of $H$, again denoted by $\tau$.

We claim that $(H,\cdot)$ is a reduced commutative Krull monoid with class group $G$, that $H$ is $\tau$-$v$-simple, and that the induced automorphism $\tau_*$ of $\cC(H)$ is the identity.
Moreover, each class of $\cC(H)$ contains $\card{\Omega}$ divisorial prime ideals.
That $H$ is a reduced commutative Krull monoid with class group $G$ follows from \cite[Proposition 2.5.1.4]{ghk06}.
It also follows that the inclusion $\iota\colon H \hookrightarrow D$ is a divisor theory.
Hence, $\pone(H) = \{\, (x,g)D \cap H \mid x \in \Omega,\; g \in G \,\}$.
By construction, $\tau$ does not fix any finite nonempty subset of $\pone(H)$, and hence $H$ is $\tau$-$v$-simple.
On the other hand, $\psi(\tau(x,g)) = \psi((x,g)) = g$, so that $\tau_*$ acts trivially on $\cC(H)$.
\end{proof}

\begin{remark}
  Let $H$ be a reduced commutative Krull monoid.
  We note that it is easy to determine $\Aut(H)$.
  Let $\tau \in \Aut(H)$.
  Then $\tau$ induces an automorphism $\tau_*$ of $\cI_v^*(H)$ and further an automorphism of $\cC(H)$, that we denote by $\tau_*$ again.
  For $g \in \cC(H)$, denote by $\pone(H)(g) = \{\, \fp \in \pone(H) \mid [\fp]=g \,\}$ the nonempty divisorial prime ideals in class $g$.
  For all $\fp \in \pone(H)$, it holds that $[\tau_*(\fp)] = \tau_*([\fp])$.
  In particular, if $g$ and $h \in \cC(H)$ lie in the same $\tau_*$-orbit, then $\card{\pone(H)(g)} = \card{\pone(H)(h)}$.
  The automorphism $\tau$ is uniquely determined by the induced $\tau_* \in \Aut(\cC(H))$ as well as the family of bijections $\pone(H)(g) \to \pone(H)(\tau_*(g))$ induced by $\tau_*$.

  Conversely, suppose that $\alpha$ is an automorphism of $\cC(H)$ such that for all $g \in \cC(H)$, $\card{\pone(H)(g)}=\card{\pone(H)(\alpha(g))}$.
  For each class $g \in \cC(H)$, let $\beta_g \colon \pone(H)(g) \to \pone(H)(\alpha(g))$ be a bijection.
  Then there exists a (uniquely determined) automorphism $\tau \in \Aut(H)$ with $\tau_*(\fp) = \beta_g(\fp)$ for all $g \in \cC(H)$ and $\fp \in \pone(H)(g)$.

  In particular, we obtain the following strengthening of \cref{t-ex-mon-aut}:
  If $H$ is a commutative Krull monoid such that each class contains either zero or infinitely many divisorial prime ideals, then there exists a $\tau \in \Aut(H)$ such that $H$ is $\tau$-$v$-simple and $\tau_*$ is the identity on $\cC(H)$.
\end{remark}

Let $P$ be a set, $\cF(P)$ the (multiplicatively written) free abelian monoid with basis $P$, and $G=\quo(\cF(P))$ the free abelian group with basis $P$.
Since every element of $G\cong \bZ^{(P)}$ has finite support, any total order on $P$ induces a total order on $G$ by means of the lexicographical order and the natural total order on $\bZ$.
Explicitly, for $a \in G \setminus \{1\}$ we define $a \ge 1$ if and only if $\val_p(a) \ge 0$ for $p = \max\supp(a)$.
With respect to any such order, $G$ is a totally ordered group.
If $(G,\cdot,\le)$ is a totally ordered group, we set $G_{>1} = \{\, a \in G \mid a > 1 \,\}$ and $G_{\ge 1} = \{\, a \in G \mid a \ge 1 \,\}$.

\begin{lemma} \label{l-ex-order}
  \begin{enumerate}
    \item\label{l-ex-order:set} Let $P$ be a set and let $\tau\colon P \to P$ be a permutation having no finite orbits.
      Then there exists a total order $\le$ on $P$ such that $\tau$ is order-preserving with respect to $\le$.
      Moreover, $\tau(x) > x$ for all $x \in P$.
    \item\label{l-ex-order:group}
      Let $(P,\le_P)$ be a totally ordered set.
      Let $\tau\colon P \to P$ be a permutation such that $\tau$ is order-preserving and $\tau(x) >_P x$ for all $x \in P$.
      Let $G = \quo(\cF(P))$,  $\overline\tau \in \Aut(G)$ with $\overline\tau|_P=\tau$, and let $\le$ be the total order on $G$ induced by $\le_P$.
      Then $\overline\tau(a) > a$ for all $a \in G_{>1}$.
      In particular, $\overline\tau$ is order-preserving and $\overline\tau(G_{>1}) \subset G_{>1}$.
  \end{enumerate}
\end{lemma}

\begin{proof}
  \ref*{l-ex-order:set}
  For $x \in P$, let $x^\tau = \{\, \tau^n(x) \mid n \in \bZ \,\}$ be its $\tau$-orbit.
  Since $x^\tau$ is infinite, it is naturally totally ordered by $\tau^m(x) \le \tau^n(x)$ if and only if $m \le n$.
  Fix an arbitrary total order on the set of all $\tau$-orbits.
  For $x$, $y \in P$, define $x \le y$ if and only if either $x^\tau < y^\tau$, or if $x^\tau=y^\tau$ and there exists an $n \in \bN_0$ such that $y = \tau^n(x)$.
  Then $\le$ is a total order on $P$, and $\tau$ is order-preserving with respect to this order.
  Moreover, $\tau(x) > x$ for all $x \in P$.

  \ref*{l-ex-order:group}
  As already observed, $(G, \cdot, \le)$ is a totally ordered group.
  Let $a \in G$ with $a > 1$.
  We show $\overline\tau(a) > a$.
  We have $a = p_1^{n_1}\cdots p_r^{n_r}$ with $r \in \bN$, pairwise distinct $p_1$,~$\ldots\,$,~$p_r \in P$ and $n_1$,~$\ldots\,$,~$n_r \in \bZ^\bullet$.
  Using the total order on $P$, we may assume $p_1 > \cdots > p_r$.
  Since $a > 1$, we have $n_1 > 0$.
  Now, $\overline\tau(a) = \overline\tau(p_1)^{n_1} \cdots \overline\tau(p_r)^{n_r}$.
  Since $\overline\tau$ is order-preserving with respect to $\le_P$, we have $\overline\tau(p_1) > \cdots > \overline\tau(p_r)$ and moreover $\overline\tau(p_1) > p_1$.
  From the way we defined the total order on $G$, it follows that $\overline\tau(a) > a > 1$.
  In particular, $\overline\tau$ is order-preserving and $\overline\tau(G_{>1}) \subset G_{>1}$.
\end{proof}

If $\cF(P)$ is a free abelian monoid and $\overline\tau$ is an automorphism of $\cF(P)$ which has no finite orbits on $P$, then \cref{l-ex-order} implies that the quotient group $\quo(\cF(P))$ admits the structure of a totally ordered group with respect to which $\overline\tau$ is order-preserving, etc.
The following is a strengthening of this result to quotient groups of reduced commutative Krull monoids.

\begin{prop} \label{p-grp-order}
 Let $H$ be a reduced commutative Krull monoid and let $\tau \in \Aut(H)$ be such that $H$ is $\tau$-$v$-simple.
 Let $G$ denote the quotient group of $H$, and denote the extension of $\tau$ to $\Aut(G)$ again by $\tau$.
 Then there exists an order $\le$ on $G$ such that $(G,\cdot,\le)$ is a totally ordered group, $H \subset G_{\ge 1}$, and $\tau(a) > a$ for all $a \in G_{> 1}$.
 In particular, $\tau$ is order-preserving on $G$ and $\tau(G_{>1}) \subset G_{>1}$.
\end{prop}

\begin{proof}
  Since $H$ is a commutative Krull monoid, it has a divisor theory.
  Because $H$ is reduced, this divisor theory can be taken to be an inclusion.
  Thus, explicitly, there exists a set $P$ such that $\iota\colon H \hookrightarrow F=\cF(P)$, $G \subset \quo(F)$, and such that the inclusion $\iota$ induces a monoid isomorphism
  \[
  \iota^*\colon F \to \cI_v^*(H),\quad  a \mapsto aF \cap H.
  \]
  Moreover, $(\iota^*)^{-1}(aH) = a$ for all $a \in H$, and $\iota^*|_P \colon P \to \pone(H)$ is a bijection.
  (See \cite[Theorem 2.4.7.3]{ghk06}.)
  Recall that $\tau$ induces a monoid automorphism $\tau_*\colon \cI_v^*(H) \to \cI_v^*(H)$ and that $\tau_*(aH) = \tau(a)H$ for all $a \in H$.

  We first show that $\tau$ extends to an automorphism of $F$.
  Through $\iota^*$, we obtain an automorphism $\overline\tau = (\iota^*)^{-1} \circ \tau_* \circ \iota^* \in \Aut(F)$.
  But we also have $H \subset F$ via the inclusion $\iota$.
  We claim that in fact $\overline\tau|_H = \tau$.
  Let $a \in H$.
  Then
  \[
  \overline\tau(a) = (\iota^*)^{-1} \circ \tau_* \circ \iota^*(a)
           =  (\iota^*)^{-1} \circ \tau_*(aH) = (\iota^*)^{-1}(\tau(a)H) = \tau(a).
  \]
  Moreover, $\overline\tau$ extends to an automorphism of $\quo(F)$, again denoted by $\overline\tau$, and then also $\overline\tau|_G=\tau$ on $G$.

  The automorphism $\overline\tau$ induces a permutation on $P$, and \subref{l-simple:orbits} implies that $\overline\tau$ does not have any finite orbits on $P$.
  Thus, \subref{l-ex-order:set} implies that there exists a total order $\le_P$ on $P$ such that $\overline\tau|_P \colon P \to P$ is order-preserving and $\overline\tau(p) > p$ for all $p \in P$.
  Let $\le$ denote the order on $\quo(F)$ induced by $\le_P$.
  Then $(\quo(F), \cdot, \le)$ is a totally ordered group.
  By \subref{l-ex-order:group}, $\overline\tau(a) > a$ for all $a \in \quo(F)_{>1}$.
  Denote the restriction of $\le$ to $G$ again by $\le$.
  Then $(G,\cdot,\le)$ is a totally ordered group, and $\tau(a) > a$ for all $a \in G_{>1}$.
  Clearly $H \subset G_{\ge 1}.$
\end{proof}

Let $(G,\cdot,\le)$ be a totally ordered group and $K$  a field.
The group algebra $K[G]$ is naturally $G$-graded.
Using the total order on $G$, it is easy to check that $K[G]$ is a domain.
Every unit of $K[G]$ is homogeneous, that is, $K[G]^\times = \{\, \lambda g \mid \lambda \in K^\times,\; g \in G \,\}$.
It follows that every nonzero principal ideal $\fa$ of $K[G]$ has a uniquely determined generator of the form $1 + f$ with $\supp(f) \subset G_{>1}$.
We call $1+f$ the \emph{normed generator} of $\fa$.

\begin{prop} \label{p-kg-simple}
  Let $H$ be a reduced commutative Krull monoid, and let $\tau \in \Aut(H)$ be such that $H$ is $\tau$-$v$-simple.
  Let $G$ denote the quotient group of $H$ and let $K$ be a field.
  If $\varphi \in \Aut(K[G])$ with $\varphi|_H = \tau$ and $\varphi(K) \subset K$, then $K[G]$ is $\varphi$-$v$-simple.
\end{prop}

\begin{proof}
  Denote the extension of $\tau$ to $G$ again by $\tau$.
  Note that $\varphi|_G = \tau$.
  By \cref{p-grp-order}, there exists an order $\le$ on $G$ such that $(G,\cdot,\le)$ is a totally ordered group and $\tau(G_{> 1}) \subset G_{> 1}$.
  Since $G$ is a subgroup of a free abelian group, it satisfies the ascending chain condition on cyclic subgroups.
  Hence, $K[G]$ is a commutative Krull domain with trivial class group by \cref{p-krull-semigroup-domain}.
  Thus, every divisorial ideal of $K[G]$ is principal.
  To show that $K[G]$ is $\varphi$-$v$-simple it therefore suffices to show $\varphi(\fa) \ne \fa$ for all principal ideals $\fa$ of $K[G]$ with $\fa \notin \{\mathbf 0, K[G]\}$.
  Let $\fa$ be such an ideal.
  Let $f \in K[G]$ with $\supp(f) \subset G_{>1}$ be such that $1+f$ is the normed generator of $\fa$.
  Since $\fa \ne K[G]$, we have $\supp(f) \ne \emptyset$.
  Now $\varphi(1+f) = \varphi(1) + \varphi(f) = 1 + \varphi(f)$.
  Moreover, $\supp(\varphi(f)) = \tau(\supp(f)) \subset G_{>1}$.
  Hence $1+\varphi(f)$ is the normed generator of $\varphi(\fa)$.
  Since $H$ is $\tau$-$v$-simple, $\tau(\supp(f)) \ne \supp(f)$ by \cref{l-simple}.
  Thus $1+\varphi(f) \ne 1 + f$, and $\varphi(\fa) \ne \fa$.
\end{proof}

\begin{thm} \label{t-extend-simple}
  Let $D$ be a commutative Krull domain and let $\sigma \in \Aut(D)$ be such that $D$ is $\sigma$-$v$-simple.
  Let $H$ be a reduced commutative Krull monoid and let $\tau \in \Aut(H)$ be such that $H$ is $\tau$-$v$-simple.
  Let $\varphi\in \Aut(D[H])$ denote the extension of $\sigma$ and $\tau$ to $D[H]$, i.e., $\varphi|_D=\sigma$ and $\varphi|_H=\tau$.
  Then $D[H]$ is $\varphi$-$v$-simple.
\end{thm}

\begin{proof}
  Let $\varphi_*$ denote the permutation of $\pone(D[H])$ induced by $\varphi$.
  There are injective maps
  \begin{align*}
    \iota_D^*\colon&
    \begin{cases} \pone(D) &\to \pone(D[H]), \\
      \fp &\mapsto \fp[H],
    \end{cases}
    &
    \iota_H^*\colon&
    \begin{cases} \pone(H) &\to \pone(D[H]), \\
      \fp &\mapsto D[\fp].
    \end{cases}
  \end{align*}
  The image of $\iota_D^*$ consists of all $\fp \in \pone(D[H])$ with $\fp \cap D^\bullet \ne \emptyset$, while the image of $\iota_H^*$ consists of all $\fp \in \pone(D[H])$ with $\fp \cap H \ne \emptyset$.
  Let $K=\quo(D)$ and $G = \quo(H)$.
  The group algebra $K[G]$ is the localization of $D[H]$ by $H$ and $D^\bullet$.
  There is a bijection
  \[
  \iota_{K[G]}^* \colon
  \begin{cases}
      \{\, \fp \in \pone(D[H]) \mid \fp \cap (D^\bullet \cup H) = \emptyset \,\} &\to \pone(K[G]), \\
      \fp           & \mapsto    \fp K[G], \\
      \fP\cap D[H]  & \mapsfrom  \fP.
  \end{cases}
  \]
  In particular,
  \[
  \pone(D[H]) = \iota_D^*\pone(D) \;\cup\; \iota_H^*\pone(H) \;\cup\; (\iota_{K[G]}^*)^{-1}\pone(K[G]).
  \]
  All of this follows from \cite[Chapter III, Sections 15 and 16]{gilmer84}, together with the fact that nontrivial essential discrete valuation overmonoids (overrings) of commutative Krull monoids (domains) bijectively correspond to nonempty (nonzero) divisorial prime ideals.
  That $\iota_D^*$ takes the stated form follows from \cite[Theorem 15.3]{gilmer84}, and the corresponding fact for $\iota_H^*$ is a consequence of \cite[Theorem 15.7]{gilmer84}.
  The stated decomposition of $\pone(D[H])$ follows from \cite[Corollary 15.9]{gilmer84}.

  Since $\varphi(H) = H$ and $\varphi(D) = D$, each of the sets $\iota_D^*\pone(D)$, $\iota_H^*\pone(H)$, and $(\iota_{K[G]}^*)^{-1}\pone(K[G])$ is fixed by $\varphi_*$.
  To show $\varphi_*(X) \ne X$ for all finite $\emptyset \ne X \subset \pone(D[H])$, it therefore suffices to consider subsets of each of these three sets.
  If $X \subset \iota_D^*\pone(D)$, then $\varphi_*(X) \ne X$, since $D$ is $\sigma$-$v$-simple.
  If $X \subset \iota_H^*\pone(H)$, then $\varphi_*(X) \ne X$, since $H$ is $\tau$-$v$-simple.

  Finally, consider the case where $X \subset (\iota_{K[G]}^*)^{-1}\pone(K[G])$.
  Since $\varphi(H) \subset H$ and $\varphi(D^\bullet) \subset D^\bullet$, $\varphi$ extends to an automorphism of $K[G]$, which we again denote by $\varphi$.
  It now suffices to show that $K[G]$ is $\varphi$-$v$-simple.
  However, this follows from \cref{p-kg-simple}.
\end{proof}

The following \subref{l-loc:basic} is a slight reformulation of the original localization argument of Claborn, which can be found in  \cite[Theorem 14.2]{fossum73} and \cite[Theorem 7]{claborn66}.
Since we need to observe some details in the argument, we give the proof anyway.
Recall that if $D$ is a commutative domain and $a$,~$b \in D$ are coprime (that is, $aD \cap bD = abD$), then $aX + b$ is a prime element of $D[X]$ (see \cite[Lemma 14.1]{fossum73}).

\begin{lemma} \label{l-loc}
  Let $D$ be a commutative Krull domain, and let $H$ be a commutative Krull monoid containing a countable set $P$ of non-associated prime elements, so that $H = H_0 \times \cF(P)$ with a commutative Krull monoid $H_0$.
  Suppose that $D[H]$ is a Krull domain.
  \begin{enumerate}
    \item\label{l-loc:basic} There exists a multiplicative subset $S \subset D[H]^\bullet$ such that $S$ is generated by prime elements of $D[H]$, $S \cap D[H_0] = \emptyset$, and $S^{-1}D[H]$ is a Dedekind domain but not a field.
    \item\label{l-loc:auto} If $\varphi \in \Aut(D[H])$ with $\varphi(P) \subset P$, then $S$ can be chosen in such a way that $\varphi(S) \subset S$.
    \item\label{l-loc:prop}
      Let $S$ be a multiplicative subset of $D[H]^\bullet$ such that $S^{-1}D[H]$ is a Dedekind domain.
      Let $\varphi \in \Aut(D[H])$ be such that $D[H]$ is $\varphi$-$v$-simple and  $\varphi(S) \subset (S^{-1}D[H])^\times$.
      Then $\varphi$ extends to an automorphism $\varphi_S \in \Aut(S^{-1}D[H])$ and $S^{-1}D[H]$ is $\varphi_S$-simple.
  \end{enumerate}
\end{lemma}

\begin{proof}
\ref*{l-loc:basic}
We have $D[H]\cong D[H_0][\cF(P)] = D[H_0][\ldots,X_{-1},X_0,X_1,\ldots]$.
Let $\fP \in \spec(D[H])$ with $\height(\fP) > 1$, and let $a_\fP \in \fP^\bullet$.
Let $\fp_1$, $\ldots\,$,~$\fp_r$ be the divisorial prime ideals of $D[H]$ that contain $a_\fP$.
By prime avoidance, there exists an element $b_\fP \in \fP \setminus (\fp_1 \cup \ldots \cup \fp_r)$.
Let $X_\fP \in \{ \ldots, X_{-1},X_0,X_1, \ldots \}$ be such that $X_\fP$ is not contained in the support of $a_\fP$ or $b_\fP$.
Then $D[H] = R_0[X_\fP]$ with $R_0=D[H_0][\{\, X_i \mid i \in \bZ,\; X_i \ne X_\fP \,\}]$ and $a_\fP$,~$b_\fP \in R_0$ are coprime.
Hence $f_\fP = a_\fP X_\fP + b_\fP$ is a prime element of $D[H]$, and $f_\fP \in \fP$.
By construction, $f_\fP \notin D[H_0]$.

Let $Q = \{\, f_\fP \mid \fP \in \spec(D[H]),\, \height(\fP) > 1 \,\}$ and let $S$ be the multiplicative set generated by $Q$.
Since $\spec(S^{-1}D[H])$ is in bijection with $\{\, \fp \in \spec(D[H]) \mid \fp \cap S \ne \emptyset \,\}$, it follows that $S^{-1}D[H]$ is a Krull domain of dimension at most $1$, i.e., a Dedekind domain.
Moreover, $S \cap D[H_0] = \emptyset$.

If $D[H_0]$ is not a field, then neither is $S^{-1}D[H]$.
It only remains to consider the, degenerate, special case where $D[H_0]$ is a field, i.e., $H_0$ is the trivial monoid and $D=K$ is a field.
Then $D[H] = K[\ldots,X_{-1},X_0,X_1,\ldots]$ is a polynomial ring in countably many indeterminates.
By construction, $Q$ only contains elements with $X_i$-degree equal to $1$ for some $i \in \bZ$.
However, $D[H]$ contains prime elements which are not of this form (e.g., $X_1^2 + X_0^2 X_1 + X_0$).

\ref*{l-loc:auto}
If $\fP \in \spec(D[H])$ with $\height(\fP) > 1$, then also $\height(\varphi(\fP))>1$.
Thus $\varphi$ induces a permutation of prime ideals of height greater than $1$.
Denote a set of representatives for the orbits by $\Omega$.
For each $\fP$ in $\Omega$, choose $f_\fP$ as in \ref*{l-loc:basic}.
For all $n \in \bZ$, $\varphi^n(f_\fP)$ is a prime element contained in $\varphi^n(\fP)$.
Since $\varphi(P) \subset P$, we have $\varphi^n(f_\fP) \notin D[H_0]$.
Set $Q = \bigcup_{\fP \in \Omega} \bigcup_{n \in \bZ} \varphi^n(f_\fP)$, and let $S$ be the multiplicative subset of $D[H]^\bullet$ generated by $Q$.
Since $\varphi(Q) \subset Q$, also $\varphi(S) \subset S$.
Thus, $S$ has the stated properties.

\ref*{l-loc:prop}
Since $\varphi(S) \subset (S^{-1}D[H])^\times$, $\varphi$ extends to an automorphism $\varphi_S$ of $S^{-1}D[H]$.
Localization induces a bijection between $\{\, \fp \in \pone(D[H]) \mid \fp \cap S = \emptyset \,\}$ and $\pone(S^{-1}D[H])$.
Hence $S^{-1}D[H]$ is $\varphi_S$-$v$-simple.
Since $S^{-1}D[H]$ is a Dedekind domain, every ideal is divisorial.
Thus $S^{-1}D[H]$ is $\varphi_S$-simple.
\end{proof}

\begin{remark}
  Most of the technicalities in the previous proof can be avoided as long as $\cC(H)$ is non-trivial and we are not picky about whether or not $S \cap D[H_0] = \emptyset$.
  In this case, we take $S$ to be the multiplicative set generated by all prime elements of $D[H]$.
  Claborn's argument shows that each $\fP \in \spec(D[H])$ with $\height(\fP) > 1$ contains some prime element, so that indeed $\dim(S^{-1}D[H]) \le 1$.
  We have $\varphi(S) \subset S$, since prime elements are mapped to prime elements by $\varphi$.
  And, finally, since $\cC(H)$ is non-trivial, there must exist a non-principal divisorial prime ideal $\fp \in D[H]$.
  Then $\fp \cap S= \emptyset$, hence $\dim(S^{-1}D[H]) = 1$.
\end{remark}

\begin{thm} \label{t-extend}
  Let $D$ be a commutative Krull domain and let $\sigma \in \Aut(D)$ be such that $D$ is $\sigma$-$v$-simple.
  Let $H$ be a reduced commutative Krull monoid containing prime elements,
  and let $\tau$ be an automorphism of $H$ such that $H$ is $\tau$-$v$-simple.
  Let $\varphi\colon D[H]\to D[H]$ denote the extension of $\sigma$ and $\tau$ to $D[H]$, that is, $\varphi|_D = \sigma$ and $\varphi|_H = \tau$.
  Let $p \in H$ be a prime element and let $p^\tau = \{\, \tau^n(p) \mid n \in \bZ \,\}$ be its $\tau$-orbit, so that $H = H_0 \times \cF(p^\tau)$ for a Krull monoid $H_0$.
  \begin{enumerate}
    \item \label{t-extend:exist} There exists a multiplicative subset $S$ of the semigroup algebra $D[H]$ such that $S$ is generated by prime elements, $S \cap D[H_0] = \emptyset$, $\varphi(S) \subset S$, and the localization $S^{-1}D[H]$ is a Dedekind domain but not a field.
    \item \label{t-extend:clsgrp}
      Let $S \subset D[H]^\bullet$ be a multiplicative subset such that $R=S^{-1}D[H]$ is a Dedekind domain but not a field and $\varphi(S) \subset R^\times$.
      Let $A$ be the subgroup of $\cC(D) \times \cC(H) = \cC(D[H])$ generated by classes of $\fp \in \pone(D[H])$ with $\fp \cap S \ne \emptyset$.
      Then $\varphi$ extends to an automorphism $\varphi_S$ of $R$, the skew Laurent polynomial ring $T=R[x,x^{-1};\varphi_S]$ is a noncommutative simple Dedekind domain, and the following sequence of abelian groups is exact:
      \[
      \xymatrix@C=0.75cm{
        \cC(D) \times \cC(H)/A \ar[rr]^{\id - (\sigma_*,\tau_*)}
        & & \cC(D) \times \cC(H)/A \ar[r]^{\beta}
        & G\big(R[x,x^{-1};\varphi_S]\big) \ar[r]
        & \mathbf 0.
      }
      \]
      Here, $(\sigma_*,\tau_*)$ is the automorphism of $\cC(D) \times \cC(H)/A$ induced by $\sigma$ and $\tau$.
      The map $\beta$ is induced as follows: If $\fa \in \cF_v(D)^\times$, the class of $\fa$ is mapped to $\kcls{\fa[H] \otimes_R T} - \kcls{T}$.
      If $\fb \in \cF_v(H)^\times$, the class of $\fb$ is mapped to $\kcls{D[\fb] \otimes_R T} - \kcls{T}$.
  \end{enumerate}
\end{thm}

\begin{proof}
  \ref*{t-extend:exist}
  Since $\tau$ does not have finite orbits on $\pone(H)$, the orbit $p^\tau$ of $p$ consists of countably many non-associated prime elements.
  Let $S$ be a multiplicative subset of $D[H]$ as in \subref{l-loc:auto}, where we take $P=p^\sigma$.

  \ref*{t-extend:clsgrp}
  By \cref{t-extend-simple}, $D[H]$ is $\varphi$-$v$-simple.
  By \subref{l-loc:prop}, $\varphi$ extends to an automorphism $\varphi_S \in \Aut(R)$, and $R$ is $\varphi_S$-simple.
  By Nagata's Theorem and the identifications we have made, $\cC(D) \times \cC(H)/A \cong \cC(R)$ with the isomorphism given by $([\fa]_D,[\fb]_H) +A \mapsto [S^{-1}\fa[\fb]]$.
  Since $R$ is $\varphi_S$-simple, $T = R[x,x^{-1};\sigma]$ is a noncommutative simple Dedekind domain by \cref{p-nc-dedekind}.
  Under the identification of $\cC(D) \times \cC(H) / A$ with $G(R)$, the automorphism $(\sigma_*,\tau_*)$ corresponds to $\varphi_*$.
  The exact sequence of class groups follows from \cref{p-k0-ext} and the discussion that followed it.
\end{proof}

\begin{remark} \label{r-primes-techn}
  The technical condition that $H$ contains a prime element (and hence, since $\tau$ does not have any finite orbits on $\pone(H)$, infinitely many non-associated ones) is necessary so that $D[H]$ has the form $D[H_0][\ldots,X_{-1},X_0,X_1,\ldots]$ with $\varphi$ acting by $\varphi(X_i) = X_{i+1}$.
  The countably many indeterminates are used to construct the prime elements which generate $S$, see \cref{l-loc}.
  (See \cite[Proposition 14]{chang11} for a refinement that only needs one indeterminate.)
  If $H$ does not contain a prime element, we may replace $H$ by $H'=H \times \cF(\ldots,p_{-1},p_0,p_1,\ldots)$ and extend $\tau$ by $\tau(p_i)=p_{i+1}$.
  Then $H'$ satisfies the conditions of the theorem, and $\cC(H)\cong \cC(H')$.
  By formulating the theorem in the slightly more technical way, we avoid the need to enlarge $H$ if it already contains prime elements.
\end{remark}

\begin{proof}[Proof of Theorem 1.1]
  We assume without restriction that $\kappa$ is infinite.
  Let $G$ be an abelian group.
  \Cref{t-ex-mon-aut} implies that there exist a reduced commutative Krull monoid with $\cC(H) \cong G$ and an automorphism $\tau$ of $H$ such that $H$ is $\tau$-$v$-simple, $\tau_* \colon \cC(H) \to \cC(H)$ is the identity, and each divisorial ideal class of $H$ contains $\kappa$ nonempty divisorial prime ideals.
  Let $K$ be a field.
  Then $K$ is simple, and hence $\id_K$-$v$-simple.
  Let $\varphi\colon K[H] \to K[H]$ be the automorphism of $K[H]$ with $\varphi|_H = \tau$ and $\varphi|_K=\id_K$.
  Let $P \subset H$ be a countable set of prime elements such that $H\setminus P$ still contains $\kappa$ prime elements.
  Then $H = H_0 \times \cF(P)$ with a Krull monoid $H_0$.
  Each class of $\cC(H_0)$ contains $\kappa$ divisorial prime ideals.
  Applying \cref{t-extend}, we find a subset $S \subset K[H]$ such that $S \cap K[H_0] = \emptyset$, the localization $R=S^{-1}K[H]$ is a commutative Dedekind domain but not a field, and $T = S^{-1}K[H][x,x^{-1};\varphi_S]$ is a noncommutative simple Dedekind domain with $G(T) \cong G$.

  If $\fp \in \pone(H)$, then $K[\fp] \in \pone(K[H])$.
  If $\fp \notin \{\, (p) \mid p \in P \,\}$, then $K[\fp] \cap S = \emptyset$ by construction.
  In this case, $S^{-1}\fp$ is a nonzero prime ideal of $R$.
  Thus, each ideal class of $R$ contains at least $\kappa$ nonzero prime ideals.
  If $\fq$ is a nonzero prime ideal of $R$, then $\fq T$ is a maximal right ideal of $T$ by \cite[Lemma 6.9.15]{mcconnell-robson01}.
  Since $T$ is flat over $R$, we have $\fq T \cong \fq \otimes_R T$.
  If $\fp \in \pone(H)$, the isomorphism $\beta\colon \cC(H) \to G(T)$ maps $[\fp]$ to $\kcls{K[\fp] \otimes_R T} - \kcls{T}$.
  It follows that each class of $G(T)$ contains at least $\kappa$ maximal right ideals.
\end{proof}

\begin{remark}
  \begin{enumerate}
    \item We can only give a lower bound on the cardinality of maximal right ideals in each class.
      Apart from the divisorial prime ideals of the form $K[\fp]$ with $\fp \in \pone(H)$, additional divisorial prime ideals arise from prime elements of $K[\quo(H)]$.
      In \cite{chang11}, Chang has shown that if $D[H]$ is a commutative Krull domain and $H$ is non-trivial, then each divisorial ideal class contains a nonzero divisorial prime ideal.

    \item
      A domain $D$ is \emph{half-factorial} if every element of $D^\bullet$ can be written as a product of irreducibles and the number of irreducibles in each such factorization is uniquely determined.
      It is conjectured that every abelian group is the class group of a half-factorial commutative Dedekind domain.
      See \cite[\S5]{gilmer06} for background.
      The conjecture is equivalent to one purely about abelian groups (\cite[Proposition 3.7.9]{ghk06}).
      See \cite{geroldinger-goebel03} for progress on this question.
  \end{enumerate}
\end{remark}

\noindent
\textbf{Acknowledgments.}
\phantomsection
\addcontentsline{toc}{section}{Acknowledgments}
I thank Alfred Geroldinger for feedback on a preliminary version of this paper.

\bibliographystyle{hyperalphaabbr}
\bibliography{all}

\end{document}